\documentclass[reqno,10pt]{amsart}
\usepackage[all,line,poly,rotate,ps,dvips]{xy}
\SelectTips{cm}{}

\usepackage{comment}
\usepackage{amssymb,eucal,graphicx}
\usepackage{enumerate}
\usepackage {paralist}
 \usepackage[percent]{overpic}

\numberwithin{equation}{section}

\newtheorem{theorem}{Theorem}[section]

\newtheorem{conjecture}[theorem]{Conjecture}

\newtheorem{thm}[theorem]{Theorem}
\newtheorem{cor}[theorem]{Corollary}

\newtheorem{prop}[theorem]{Proposition}
\theoremstyle{definition}

\newtheorem{example}[theorem]{Example}

\theoremstyle{remark}
\newtheorem{rem}[theorem]{Remark}

\begin{document}

\newcommand{\Gal}{{\rm Gal}} 
\newcommand{\Pic}{{\rm Pic}}


\addtolength{\textwidth}{0mm}
\addtolength{\hoffset}{-0mm} 
\addtolength{\textheight}{0mm}
\addtolength{\voffset}{-0mm}

\title{On finiteness of curves with high canonical degree on  a surface}

\author{Ciro Ciliberto, Xavier Roulleau}
\address{Universit\`a degli Studi di Roma Tor Vergata, Via della
  Ricerca Scientifica 1, 00133, Rome,
  Italy.}\email{cilibert@axp.mat.uniroma2.it}
  
  \address{Laboratoire de Math\'ematiques et Applications de lÕuniversit\'e de Poitiers,
  T\'el\'eport 2 -- BP 30179, Boulevard Marie et Pierre Curie, 86962 Futuroscope Chasseneuil Cedex,
  France.}\email{Xavier.Roulleau@math.univ-poitiers.fr}

\begin{abstract} The \emph{canonical degree} of a curve $C$ on a surface $X$ is 
$K_X\cdot C$. Our main result, Theorem \ref {thm:MAIN}, is that on a surface of general type there are only finitely many curves
with negative self--intersection and sufficiently large canonical degree. Our proof strongly relies on results by Miyaoka. We extend our result both to surfaces not of general type and to non--negative curves, and give applications, e.g. to finiteness of negative curves on a general blow--up of $\mathbb P^ 2$ at $n\geq 10$ general points (a result related to  \emph{Nagata's Conjecture}).  We finally discuss a conjecture by Vojta concerning the asymptotic behaviour of the ratio between the canonical degree and the geometric genus of a curve varying on a surface. 
The results in this paper go in the direction of understanding the \emph{bounded negativity} problem. 
\end{abstract}
\maketitle

\section{Introduction.}

Let $C$ be a projective curve on a smooth projective complex surface $X$. By \emph{curve} we mean an irreducible, reduced 1--dimensional scheme. We denote by $g=g(C)$
its geometric genus and by $p=p_{a}(C)$ its arithmetic genus, i.e.
$C^{2}+K\cdot C=2p_{a}-2$, where $K=K_{X}$ is a canonical divisor of
$X$. We set $\delta=\delta(C)=p-g$.
We call a curve \emph{negative} if $C^{2}<0$. The \emph{canonical degree} of $C$ is $k_C=K\cdot C$, often simply denoted by $k$. 
If $g(C)\neq 1$, we  set
\[
\beta_C=\frac{k_C}{g(C)-1},
\]
 often simply denoted by $\beta$. 
For a surface $X$ we  set 
\[a_X=3c_{2}(X)-K_{X}^{2},\] 
often simply  denoted by $a$.  If the Kodaira
dimension $\kappa=\kappa(X)$ is non--negative, one has $a\geq 0$.

The main result of this paper concerns negative curves with high canonical degree:

\begin{theorem}
\label{thm:MAIN} (A) Let $C$  be a negative curve 
not isomorphic to $\mathbb{P}^{1}$ on a surface $X$ with $\kappa\geq0$. 
If $a=0$, there are no rational curves on $X$, i.e. 
either $a\not = 0$ or $g>0$. Moreover 
\begin{equation}\label{eq:1}
k_C\leq3(g-1)+\frac{3}{4}a+\frac{1}{4}\sqrt{9a^{2}+24a(g-1)}.
\end{equation}
Furthermore, if $g>1$ then 
 \begin{equation}\label{eq:3}
\beta_C \leq3+\frac{3}{4}a+\frac{1}{4}\sqrt{9a^{2}+24a}\leq4+\frac{3}{2}a.
\end{equation}
If, in addition, $\beta=3+\epsilon>3$, then
\begin{equation}\label{eq:4}
g\leq1+\frac{3a(\epsilon+1)}{2\epsilon^{2}}.
\end{equation}
(B) Suppose $\kappa(X)=2$. Then for any $\epsilon>0$
there are at most finitely many negative curves $C$ on $X$ such that 
 $k_C\geq(3+\epsilon)(g-1)>0$.
\end{theorem}
   
From (B) it follows that: 

\begin{cor}
Let $X$ be a surface of general type. There  is a function $B(\epsilon), $ defined for $\epsilon \in ]0,\infty[$ such that for all negative curves $C$ we have 
\[
k_C\leq(3+\epsilon)(g-1)+B(\epsilon) \,\,\, \text{and}\,\,\, -C^{2}\leq(1+\epsilon)(g-1)+B(\epsilon).
\]

\end{cor}

On some smooth quaternionic Shimura surfaces $X$ there are infinitely many totally geodesic curves
(see \cite {Chinburg}). Such a curve $C$ is
also called a \emph{Shimura curve} and satisfies $k_C=4(g-1)$. We thus obtain the following corollary, 
which was one of the main result of
\cite{BR}:

\begin{cor} On a Shimura surface, there exist finitely many (may be none) negative
Shimura curves.
\end{cor}

The proof of Theorem \ref {thm:MAIN}, contained in \S \ref {sec:proof}, strongly relies on a result by 
Miyaoka's (see  \cite[Cor 1.4]{Miyaoka} stated as Theorem \ref {thm:(Miyaoka)} below).  
In particular, the inequality \eqref{eq:1} is very similar to  \cite[formula (3)]{Miyaoka}, which has a slightly lower growth in $g$, but applies only to minimal surfaces.

In \S \ref {sec:nagata} we make an extension of Theorem \ref {thm:MAIN} which works also in the case $\kappa=-\infty$, and we prove a finiteness result for negative curves on a general blow--up of $\mathbb P^ 2$ at $n\geq 10$ general points. This is  a bounded negativity result which is 
reminiscent of the famous \emph{Nagata's Conjecture}, predicting that there is no negative curve on such a surface except for $(-1)$--rational curves. 

In \S \ref {sec:pos}, using again   Miyaoka's result, we prove a boundedness theorem  for non--negative curves of high canonical degree. In \S \ref {sec:voita} we discuss a conjecture by Vojta concerning the asymptotic behaviour of $k_C/(g-1)$ when $C$ varies among all curves on a surface. We introduce an invariant related to Vojta's conjecture and we prove a bound for it.

The results in this paper go in the direction of understanding \emph{bounded negativity} (see \cite {BR}). The \emph{Bounded Negativity Conjecture} (BNC)
predicts that on a surface of general type over $\mathbb{C}$ 
the self--intersection of negative curves is bounded below. Nagata's conjecture, which we mentioned above, is also a sort of bounded negativity assertion.  As a general reference on both bounded negativity and Nagata's conjecture, see \cite{Harbourne}. Also Vojta's conjecture is related to bounded negativity, as we discuss in \S \ref {sec:voita}.
As a consequence of Theorem \ref {thm:MAIN}, we have the following information on negative curves:

\begin{cor}
Suppose BNC fails, so that there
exists a sequence $(C_{n})_{n \in \mathbb N}$ of negative curves of genus $g_{n}$
with $\lim g_{n}=\infty$ and $\lim C_{n}^{2}=-\infty$. Then 
\[
\limsup_n\frac{K\cdot C_{n}}{g_{n}-1}\leq3.
\]
\end{cor}

In conclusion, the authors would like to thank B. Harbourne and J. Ro\'e for useful exchanges of ideas about the application to Nagata's Conjecture in \S \ref {sec:nagata}.

\section{The proof of the main theorem} \label{sec:proof}

Our proof relies on the following result by Miyaoka's 
(see  \cite[Cor 1.4]{Miyaoka}):

\begin{thm}
\label{thm:(Miyaoka)} Let $C$ be curve on a surface $X$ with $\kappa\geq 0$.  Then for all $\alpha\in[0,1]$,
we have: 
\begin{equation}
{\alpha^{2}}(C^{2}+3k_C-6g+6)-4\alpha(k_C-3g+3)+2a\geq0.\label{eq:Miyaoka 1}
\end{equation}
Suppose $C$ is not isomorphic to $\mathbb{P}^{1}$ and 
$k_C>3(g-1)$. Then
\begin{equation}
2(k_C-3g+3)^{2}-a(C^{2}+3k_C-6g+6)\leq0.\label{eq:Miyaoka 1bis}
\end{equation}
Suppose  in addition $K^{2}>0$. Then
\begin{equation}
(\frac{c_{2}}{K^{2}}-1)k_C^{2}+(4(g-1)+a)k_C-2(g-1)(3(g-1)+a)\geq(\frac{c_{2}}{K^{2}}-\frac{1}{3})[k_C^{2}-C^{2}K^{2}]\geq0.\label{eq:Miyaoka 2}
\end{equation}
\end{thm}
\begin{proof}
Inequality \eqref {eq:Miyaoka 1} is \cite[Thm 1.3, (i)]{Miyaoka} and \eqref {eq:Miyaoka 1bis} is \cite[Thm 1.3, (ii)]{Miyaoka}.
As for \eqref {eq:Miyaoka 2}
this is \cite[Thm 1.3, (iii)]{Miyaoka}, which is stated there under the
assumption that $X$ is minimal of general type and $C\not\simeq\mathbb{P}^{1}$.
However Miyaoka's argument works more generally  under the weaker assumption $K^{2}>0$. 
\end{proof}

We are now ready for the: 
\begin{proof}[Proof of Theorem \ref{thm:MAIN}]
Let us prove (A). Let $C$ be a negative curve on $X$
 not isomorphic to $\mathbb{P}^{1}$.  Then
$-aC^{2}\ge  0$, with equality if and only if $a=0$. If $k_C\leq3(g-1)$, there is nothing
to prove. Let us suppose $k=k_C>3(g-1)$ and set $\mathfrak g=g-1$. By
\eqref{eq:Miyaoka 1bis}, one has
\begin{equation}\label{eq:p}
P(k):=2(k-3\mathfrak g)^{2}-a(3k-6\mathfrak g ) \leq 0,
\end{equation}
with strict inequality if $a>0$. 

If $a=0$ and $g=0$, the polynomial $P$ is positive, thus this cannot occur. In the remaining cases,  $k_C$ is less than or equal to the largest root of  $P$, whence we get  \eqref {eq:1}.
 
 Suppose $g>1$ and let $\epsilon>0$ be such that $k_C=(3+\epsilon)(g-1)$. By \eqref {eq:1} we obtain
\begin{equation}\label{eq:big}
\epsilon(g-1)\leq\frac{3}{4}a+\frac{1}{4}\sqrt{9a^{2}+24a(g-1)},
\end{equation}
hence
\[
\epsilon\leq\frac{3a}{4(g-1)}+\frac{1}{4(g-1)}\sqrt{9a^{2}+24a(g-1)}\leq\frac{3}{4}a+\frac{1}{4}\sqrt{9a^{2}+24a}\leq\frac{3}{2}a+1.
\]
which yields \eqref {eq:3}.
On the other hand, \eqref {eq:big} reads
\[
4\epsilon(g-1)-3a\leq\sqrt{9a^{2}+24a(g-1)},
\]
and by squaring one gets \eqref {eq:4}, finishing  the proof of (A).

Next we prove (B). Let $\beta_0>3$. By \eqref {eq:3} and \eqref {eq:4}, negative curves with $\beta>\beta_0$
have bounded genus $g$, therefore by \eqref {eq:1} also $k_C$ is bounded, hence the arithmentic genus $p$ is bounded.

Suppose $K$ is big. By \cite[Cor. 2.2.7]{Lazarsfeld} there exist
$m\in\mathbb{N}^{*}$, an ample divisor $A$ and an effective divisor
$Z$ such that 
\[
mK\equiv A+Z.
\]
Since $Z$ is effective, the set of integers $Z\cdot C$, when $C$ varies among
negative curves, is bounded from below, therefore the degree $A\cdot C=(mK-Z) \cdot C$ of these curves with respect to the ample
divisor $A$ is bounded. Hence, by results of Chow--Grothendieck \cite{Grothendieck},
\cite[Lecture 15]{Mumford}, one has only finitely many components of the Hilbert scheme
containing points corresponding to such curves. 
Since they are negative, these components contain only one curve, proving the assertion. 
\end{proof}

\section{Surfaces not of general type} \label {sec:nagata}

We want to deduce from Theorem \ref {thm:MAIN} a result valid for any smooth surface. 

\begin{thm} \label{prop:nongen}Let $Y$ be any smooth projective surface. Let $\eta\in \Pic(Y)$ be such that
$\vert K_{Y}+\eta\vert$ is big and $\vert 2\eta \vert$ contains a base point free pencil.
Let $\beta_{0}>3$. Then there are finitely many negative curves  $D$ on $Y$ such that  
\begin{equation}\label{eq:nongen}
k_D \geq\beta_{0}(g-1)+\frac{\beta_{0}-2}{2}D\cdot \eta.
\end{equation} \end{thm}

\begin{proof}
Under the hypotheses there is a smooth curve $B\equiv 2\eta$ intersecting 
all negative curves of $X$ only at smooth points with intersection  multiplicity 1. Let us make a double cover $f:X\to Y$ branched along $B$.  Then for all negative curve $D$ of $X$, $C=f^ *(D)$ is irreducible, negative and  $g(C)=2g(D)-1+\eta\cdot D$ by Hurwitz formula.  Since 
$f_*(K_X)=K_Y\oplus (K_{Y}+\eta)$, then $\kappa(X)=2$ and
we finish by applying  (B) of Theorem \ref {thm:MAIN} to
$X$ and to $C=f^{*}(D)$.
\end{proof}

As an application, we take $Y_n$ to be the plane blow-up at $n$ general points.
Then ${\rm Pic}(Y_n)\cong \mathbb Z^ {n+1}$ 
 generated by the classes of the pull--back $L$ of a line and of minus the exceptional divisors $E_1,\ldots, E_{n}$ over the blown--up points. We write $D=(d,m_{1},\dots,m_{n})$ to denote the
class of a curve with components $d,m_{1},\dots,m_{n}$ with respect to this basis. 
We may use exponential notation to denote repeated  $m_i$'s. 
Thus $-K=(3, 1^ {n})$.

 \begin{prop}\label {prop:nagata}  Fix  $\beta_0>3$. There are finitely many irreducible curves of class $D=(d,m_{1},\dots,m_{n})$  on $Y_n$ such that
\begin{equation}\label{eq:nagata}
\frac {D^ 2}d \leq \frac {2-\beta_0}  {\beta_0} \big (1+ \frac M d\big )\,\,\, \text {where} \,\,\,\,\, M=\sum_{i=1}^ n m_i.
\end{equation}
\end{prop}
\begin{proof} We apply Theorem \ref {prop:nongen}, by taking $\eta=4L$. Indeed
$-K+\eta=(1,-1^ {n})$ is big.  \end{proof} 

Recall that 
\[
\epsilon_n=\inf \{\frac d M, \,\,\, \text {for all effective} \,\,\, D=(d,m_{1},\dots,m_{n}),\,\,\, \text {such that}\,\,\, M >0 \}
\]
is the \emph{Seshadri constant} of $Y_n$. Nagata's Conjecture (see \cite {Nagata}) is equivalent to say that $\epsilon _n=1/\sqrt n$ if $n\ge 10$ (see \cite {HarbourneRoe}). 

\begin{rem} Proposition  \ref {prop:nagata} can be seen as weak form of Nagata's Conjecture. 
Indeeed, let us look at the \emph{homogenous case} $D=(d,m^ n)$ with $n\geq 10$. Nagata's Conjecture predicts that, if the $n$ blown--up points are in very general position, there is no  irreducible such curve  with $D^ 2<0$ (see \cite {CM, Nagata}), i.e., with
$d<\sqrt n m$. The conclusion of Proposition \ref {prop:nagata} is not absence of curves, but finiteness of their set, under a stronger assumption than Nagata's. Let us look at the difference between the two assumptions. In the $(m,d)$--plane \eqref {eq:nagata} applies   to pairs $(m,d)$ in the first quarter below the hyperbola  with equation
\begin{equation} \label{eq:hyper}
\beta_0d^ 2+d(\beta_0-2)-n\beta_0m^ 2+(\beta_0-2)nm= 0 
\end{equation}
drawn in black in Figure \ref {fig:1}. One of its asymptotes (drawn in blue) is parallel to the \emph {Nagata line} $d=\sqrt n m$ (drawn in  green). 

\begin{figure}[h]
        \begin{center}
           \begin{overpic}[scale=0.50]{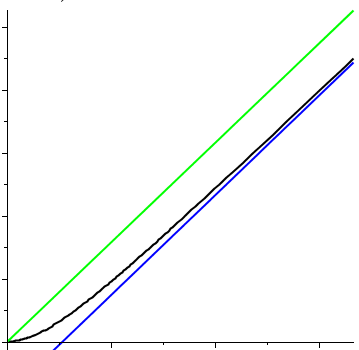}
            
                \end{overpic}
        \end{center}
                 \caption{The hyperbola, its asymptote and the Nagata line}
       \label{fig:1}
    \end{figure}

Since for all effective divisor $D=(d,m_1,\ldots,m_n)$ one has $d/M\ge \epsilon_n$, one has 
approximations $d/M\ge e_n$ to Nagata's conjecture for any lower approximation $e_n$ of $\epsilon_n$. 
The best known in general is the one in \cite  {HarbourneRoe}
\begin{equation}\label{eq:hr}
\epsilon _n \geq e_n=\sqrt {\frac 1n \big (1 - \frac 1 {f(n)}\big )}
\end{equation}
where $f(n)$ is, for most $n$, an explicitly given quadratic function of $n$ (see \cite [Corollary 1.2.3] {HarbourneRoe}). For $n=10$ in the homogeneous case the best result is $e_{10}=228/721$ (see \cite{Petra}).

The hyperbola \eqref{eq:hyper} meets the line $d=e_n m$, therefore Proposition \ref{prop:nagata} gives some information in an unlimited region where the above approximations to Nagata do not work.
\end{rem}

\begin{rem} Proposition \ref{prop:nagata} implies that there are finitely many irreducible curves of class $D=(d,m_{1},\dots,m_{n})$  on $Y_n$ such that
\begin{equation}\label{eq:sesh}
\frac {D^ 2}d \leq \frac {2-\beta_0}  {\beta_0} \big (1+ \frac 1 \epsilon_n \big),
\end{equation}
where $\epsilon_n$ can be replaced by $e_n$ in Harbourne--Ro\'e's approximation \eqref {eq:hr}. 
This result is not surprising. Indeed, J. Ro\'e pointed out to us an easy argument which shows that  there is no irreducible curve of class $D=(d,m_{1},\dots,m_{n})$  on $Y_n$ such that 
\[
\frac {D^ 2} d< - \frac 1 {n\epsilon_n}
\]
which is better than \eqref {eq:sesh}, and the difference
\[
\frac {\beta_0-2}  {\beta_0} \big (1+ \frac 1 \epsilon_n \big)- \frac 1 {n\epsilon_n}
\]
 tends to $\frac {\beta_0-2}  {\beta_0}\sim \frac 1 3$ for $n\to \infty$. \end{rem}

\section{A boundedness result for non--negative curves}\label{sec:pos}

With the usual notation, for a curve $C$ on the surface $X$ with $C^ 2\neq 0$, we set  $x_C:=\frac{\delta_C}{C^{2}}$, with the usual convention that the index $C$ can be dropped if there is no ambiguity. 

\begin{thm}\label{thm:MAIN2}
Consider real numbers  $x_{0}>\frac{1}{2}$ and $\beta_{0}>3$. Let $C$ be a curve
on $X$, with $\kappa(X)\geq 0$,  satisfying the following conditions:\\
(1)  $C^{2}>0$, $k_C=\beta(g-1)$ with $\beta>\beta_{0}$ and $g>1$;\\
(2) $x_C>x_{0}$.\\
Then 
\begin{equation}
g\leq a\frac{(\beta-2)}{(\beta-3)^{2}}\frac{3x_{0}-1}{2x_{0}-1}+1\label{eq:positif 1},
\end{equation} 
\begin{equation}
k_C\leq a\frac{(\beta-2)}{\beta(\beta-3)^{2}}\frac{3x_{0}-1}{2x_{0}-1},\label{eq:positif 2}
\end{equation}
\begin{equation}
k_C\leq 2(g-1)  +a\frac {(\beta-2)^ 2}{(\beta-3)^ 2}\cdot \frac{3x_{0}-1}{2x_{0}-1}.\label{eq:positif 3}
\end{equation}
If $\kappa(X)=2$, then the Hilbert scheme of curves  on $X$ satisfying
 (1) and (2) has finitely many irreducible components.\end{thm}
\begin{proof} One has
\[
k_C-2(g-1)=2\delta-C^ 2=(\beta-2)(g-1)>0.
\]
Hence
by \eqref{eq:Miyaoka 1}, we have
\begin{equation}
P(\alpha):=\alpha^{2}(3\delta-C^{2})+\alpha\frac{2(\beta-3)}{\beta-2}(C^{2}-2\delta)+a\geq0\,\label{eq:Positive}
\end{equation}
 for $\alpha \in [0,1]$. Since the
 coefficient of the leading term of  $P$ is positive, the minimum of $P(\alpha)$
 is attained for 
\[
\alpha_{0}=\frac{(\beta-3)(2\delta-C^{2})}{(\beta-2)(3\delta-C^{2})}.
\]
Since $\beta>3$, we have $\alpha_{0}\in ]0,1[$, and, by \eqref {eq:Positive} we have
\[
P(\alpha_0)=-\frac{(\beta-3)^{2}(2\delta-C^{2})^{2}}{(\beta-2)^{2}(3\delta-C^{2})}+a\geq0
\]
Thus
\[
\frac{a}{\mu}\geq\frac{(2\delta-C^{2})^{2}}{(3\delta-C^{2})}\,\,\, \text {where}\,\,\, \mu=\frac{(\beta-3)^{2}}{(\beta-2)^{2}},
\]
hence
\[
\frac{a}{\mu}\cdot \frac{3\delta-C^{2}}{2\delta-C^{2}}\geq2\delta-C^{2}.
\]
We have 
\[
\frac{3\delta-C^{2}}{2\delta-C^{2}}=\frac{3x-1}{2x-1}<\frac{3x_{0}-1}{2x_{0}-1}
\]
because $\frac{3x-1}{2x-1}$ is decreasing for $x>x_{0}>\frac{1}{2}$, hence
\[
(\beta-2)(g-1)=k_C-2(g-1)=2\delta-C^{2}\leq\frac{a}{\mu}\cdot \frac{3x_{0}-1}{2x_{0}-1},
\]
which implies \eqref {eq:positif 1}, \eqref {eq:positif 2} and \eqref {eq:positif 3}. 
Moreover both $g$ and  $k_C$ are bounded from above and, if $\kappa(X)=2$, we conclude with the same argument at the end of the proof of Theorem \ref {thm:MAIN}.
\end{proof}

\begin{cor}
Let be $\beta_{0}>3$ and let $(C_{n})_{n\in \mathbb  N}$ be a sequence of curves 
on
$X$ with $\kappa=2$ such that $k_{C_{n}}>\beta_{0}(g(C_{n})-1)$, $C_{n}^{2}>0$ and $\lim g(C_{n})=\infty$.
Then 
\begin{equation}\label{eq:limsup}
\lim_n \frac {\delta_{C_n} }  {C_{n}^{2}} = \frac{1}{2},
\end{equation}
moreover  $\lim_n \frac {g(C_n)} {\delta_{C_{n}}}=\lim_n \frac{K\cdot C_n}{\delta_n}=0.$
\end{cor}
\begin{proof}
Let $C$ be a curve with $k_C=(3+\epsilon)(g-1)$, $\epsilon >0$. Since $(1+\epsilon)(g-1)=2\delta-C^2$, we get $\frac{\delta}{C^2}-\frac{1}{2}=(1+\epsilon)\frac{g-1}{2C^2}\geq 0$. Therefore $\lim\inf_n \frac{\delta_{n}}{C_{n}^{2}} \geq \frac{1}{2}$. On the other hand,  by Theorem \ref{thm:MAIN2} we obtain $\lim\sup_n \frac{\delta_{n}}{C_{n}^{2}} \leq \frac{1}{2}$. The remaining limits are readily computed.
\end{proof}

\begin{example}
For Shimura curves on Shimura surfaces, we have $K\cdot C=4(g-1)$ and, if there is one, 
there are infinitely many of them, with the genus going to infinity. 
\end{example}

\section{On a conjecture by Vojta}\label{sec:voita}

The results in \S \ref {sec:pos} are reminiscent of the following conjecture (see \cite {Autissier}), which predicts that curves of bounded geometric
genus on a surface of general type form a bounded family:
\begin{conjecture}\label{conj:1}
Let $X$ be a smooth projective surface. There exist
constants $A,\, B$ such that for any curve $C$ we have
\[
k_C\leq A(g-1)+B.
\]

\end{conjecture}
If this conjecture is satisfied for  $X$ with $\kappa(X)=2$, then $X$
 contains finitely many curves of genus $0$ or
$1$. This is known to hold for minimal surfaces 
with big cotangent bundle (see \cite{Bogomolov, Deschamps}).

A stronger version of Conjecture \ref {conj:1} is the following conjecture by Vojta (see again \cite {Autissier}):
\begin{conjecture}\label{conj:2}
For any real number $\epsilon>0$, we can take $A=4+\epsilon$ in Conjecture \ref {conj:1}
(and $B=B(\epsilon)$  a function of $\epsilon$). 
\end{conjecture}

An even stronger, more recent version, predicts that $A=2+\epsilon$ (see \cite {Autissier}). 

\begin{rem}
If $C$ is a smooth curve on $X$, then $k_C=2(g-1)-C^2$, therefore if BNC holds, then Vojta's conjecture holds for smooth curves with $A=2$. This suggests a close relationship between Vojta's conjecture and BNC.
\end{rem}

Miyaoka proves in \cite {Miyaoka} that Conjecture \ref {conj:1} also holds if  $K^{2}>c_{2}$ and he gives explicit values for $A$ and
$B$,  but they are far away from the ones predicted by Conjecture \ref {conj:2}. 
Moreover Miyaoka proves that $k_C\leq 3(g-1)$ 
for (smooth) compact ball quotient surfaces on which the equality is attained by an
infinite number of curves, i.e.,  Shimura curves, if they
exists.

In \cite{Autissier} one proves 
that for surfaces whose universal cover is the bi--disk, one has 
\[
k_C\leq4(g-1).
\]
This is sharp since for Shimura curves on Shimura
surfaces, one has $k_C=4(g-1)$.

For $X$ a surface, we define 
\[
\Lambda_{X}=\sup_{(C_{n})_{n\in\mathbb{N}}} \big \{ \limsup_{n}\frac{K\cdot C_{n}}{g_{n}-1}\big \}
\]
where $(C_{n})_{n\in\mathbb{N}}$ varies among all sequences of curves
$C_{n}$ in $X$ of genus $g_{n}=g(C_{n})>1$ with $\lim_{n}g_{n}=\infty$.
Conjecture \ref {conj:1} says that $\Lambda_X<\infty$. 

If $X$ has trivial canonical bundle, then $\Lambda_X=0$. Apart form this  case, 
and the aforementioned cases  studied in  \cite{Autissier, Miyaoka},  nothing is known about $\Lambda_X$.
The following result gives us a  piece of information:

\begin{thm}\label{thm:voita} Let $X$ be a surface of general type. Let $L$ be a very ample divisor on $X$, and let $\gamma$ be the arithmetic genus of curves in $\vert L\vert$. Then
\[
\Lambda_X\geq \frac {K\cdot L}{L^ 2+\gamma-1}>0.
\]
\end{thm}

\begin{proof}  Look at the surface $X$ embedded in $\mathbb P^ r$, with $r\geq 3$, via $L$. 
Then take a general projection $\pi: X\to \mathbb P^ 2$. Consider a general rational curve  of degree $n$ in $\mathbb P^ 2$ and let $C_n$ be its pull--back via $\pi$. Then $C_n\in \vert nL\vert$. 

By Hurwitz formula, the ramification divisor $R$ of $\pi$ is such that $R\equiv K+3L$. So Hurwitz formula again, implies that the geometric  genus $g_n$ of $C_n$ 
satisfies
\[
2g_n-2=nL\cdot K+(3n-2)L^ 2.
\]
Therefore
\[
\frac {K\cdot C_n}{g_n-1}= \frac {L\cdot K}{L^ 2+\gamma-1- \frac {L^ 2}n}
\]
and this proves the assertion.\end{proof}


\begin{example} Suppose that, in the setting of  Theorem \ref {thm:voita}, one has $K=mL$, with $m>0$. Then 
\[
\Lambda_X\geq	 \frac {2m}{m+3}.
\]
So there are sequences $(X_n)_{n\in \mathbb N}$ of surfaces, e.g., complete intersections of increasing degree in projective space, with $m\to \infty$, and therefore $\Lambda_{X_n}\to 2$ (from below). 
\end{example}



\end{document}